\documentclass[11pt]{amsart}
\usepackage{amsmath,amssymb,latexsym}
\usepackage{dsfont}
\usepackage[T1]{fontenc}
\usepackage{enumerate}
\usepackage{amsfonts}
\usepackage{comment}

\def\G{\mathcal G}

\def\ls{\lesssim}
\def\t{\widetilde}

\def\Ind#1#2{#1\setbox0=\hbox{$#1x$}\kern\wd0\hbox to 0pt{\hss$#1\mid$\hss}
\lower.9\ht0\hbox to 0pt{\hss$#1\smile$\hss}\kern\wd0}

\def\Notind#1#2{#1\setbox0=\hbox{$#1x$}\kern\wd0\hbox to 0pt{\mathchardef
\nn="3236\hss$#1\nn$\kern1.4\wd0\hss}\hbox to 0pt{\hss$#1\mid$\hss}\lower.9\ht0
\hbox to 0pt{\hss$#1\smile$\hss}\kern\wd0}

\theoremstyle{plain}
\newtheorem{theorem}{Theorem}[section]
\newtheorem{prop}[theorem]{Proposition}
\newtheorem{fact}[theorem]{Fact}
\newtheorem{lemma}[theorem]{Lemma}
\newtheorem{cor}[theorem]{Corollary}
\newtheorem*{claim}{Claim}

\theoremstyle{defn}
\newtheorem{defn}[theorem]{Definition}
\newtheorem{remark}[theorem]{Remark}

\newtheorem*{observation}{Observation}

\def\Ind#1#2{#1\setbox0=\hbox{$#1x$}\kern\wd0\hbox to 0pt{\hss$#1\mid$\hss}
\lower.9\ht0\hbox to 0pt{\hss$#1\smile$\hss}\kern\wd0}

\def\Notind#1#2{#1\setbox0=\hbox{$#1x$}\kern\wd0\hbox to 0pt{\mathchardef
\nn="3236\hss$#1\nn$\kern1.4\wd0\hss}\hbox to 0pt{\hss$#1\mid$\hss}\lower.9\ht0
\hbox to 0pt{\hss$#1\smile$\hss}\kern\wd0}

\title{Groups in NTP$_2$}

\author{Nadja Hempel}
\address{ Nadja Hempel \\ Institut Camille Jordan\\
Université Claude Bernard Lyon 1\\43 boulevard du 11 novembre 1918\\
69622 Villeurbanne cedex\\
France }
\email{ hempel@math.univ-lyon1.fr}

\author{Alf Onshuus}
\address{ Alf Onshuus\\ Departamento de Matem\'aticas \\Universidad de los Andes \\
   Cra 1 No. 18A-10, Edificio H\\
   Bogot\'a, 111711\\
   Colombia}
\email{aonshuus@uniandes.edu.co}

\begin{document}
\maketitle
\begin{abstract}We prove the existence of abelian, solvable and nilpotent definable envelopes for groups definable in models of an NTP$_2$ theory. \end{abstract}

\section{Introduction}
One of the main concerns of model theory is the study of definable sets. For example,
given an abelian subgroup in some definable group, whether or not one can find a \emph{definable}
abelian group containing the given subgroup becomes very important, since it ``brings'' objects
outside the scope of model theory into the category of definable sets.

In that sense, an ongoing line of research consists of finding ``definable envelopes''. Specifically, one can ask if for a
definable group $G$ and a given subgroup of $G$ with a particular algebraic property  such as being abelian, solvable, or nilpotent, can one find a \emph{definable} subgroup of $G$ which contains the given subgroup and which has the same algebraic property.
This is always possible in stable theories (see \cite{Po}), and recent research has shown remarkable progress
both simple theories, and dependent theories:

In a dependent theory, Shelah has shown that given any definable group $G$ and any abelian subgroup of $G$, one can find a definable abelian subgroup in some extension of $G$ which contains the given abelian subgroup (see \cite{Sh783}). De Aldama generalized this result in  \cite{dAl} to nilpotent and normal solvable subgroups.

In simple theories, one cannot expect such a result to hold, as there are examples of definable groups  with simple theories which contain infinite abelian subgroups but for which all its definable abelian subgroups are finite (see Remark \ref{extraspecial}).  Nevertheless, one obtains definable envelopes ``up to finite index''. In \cite{Mi} and \cite{Mi2} Milliet proved that given any (abelian/nilpotent/solvable) subgroup $H$ of a group $G$ definable in a simple theory one can find a subgroup of $G$ which contains $H$  \emph{up to finite index} and which is (abelian/nilpotent/solvable).\footnote{The existence of nilpotent envelopes played an essential role in the proof of Palacin and Wagner showing that the ``fitting subgroup'', i.\ e.\ the group generated by all normal nilpotent subgroup, of a group definable in a simple theory is again nilpotent (see \cite{PaWa}).}

In this paper, we analyze arbitrary abelian, nilpotent and normal solvable subgroups of groups definable in theories without the tree property of the second kind (NTP$_2$ theories), which include both simple and dependent theories. We prove the existence of definable envelopes up to finite index in a saturated enough extension of a given group which is definable in a model of an NTP$_2$ theory, which is inspired by the result in simple theories as well as the one in dependent theories.

\section{Preliminaries}
In this section we state the known results in simple and dependent theories. Throughout the paper, we say that
a group is definable in a theory if the group is definable in some model of the theory. We also sometimes say that
a group is \emph{dependent}, \emph{simple} or \emph{NTP$_2$} if the theory of the group, in the language of groups, is, respectively, dependent, simple, or NTP$_2$. For some cardinal $\kappa$, a $\kappa$-saturated extension of a definable group is this group ``seen'' in an $\kappa$-saturated extension of the model in which the group is defined.

\begin{defn}
Let $M$ be a model of a theory $T$ in a language $\mathcal L$. Let $A$ be a subset of $M$.
A sequence $\langle a_i \rangle_{i\in I}$ is defined to be indiscernible over $A$, if $I$ is an ordered index set and given any formula $\phi(x_1, \dots , x_n)$
with parameters in $A$, and any two subsets $i_1< i_2< \dots i_n$ and $j_1< j_2< \dots j_n$ of $I$, we have
\[
M\models \ \phi(a_{i_1}, \dots, a_{i_n})\Leftrightarrow \phi(a_{j_1}, \dots, a_{j_n}).
\]
\end{defn}

The following is a well known fact which is proved using Erd\"os-Rado Theorem.

\begin{fact}
For some cardinal $\kappa$ and any set $A$, any sequence of elements $\langle a_i\rangle_{i\in \kappa}$ contains a subsequence of size $\omega$ which is indiscernible over $A$.

Even more, for any cardinal $\lambda$ and any set $A$, there is some cardinal $\kappa$ such that any sequence of elements $\langle a_i\rangle_{i\in \kappa}$ contains a subsequence of size $\lambda$ which is indiscernible over $A$.
\end{fact}

\begin{defn}
A theory $T$ is \emph{dependent} if in no model $M$ of $T$ one can find an indiscernible sequence $\langle \bar a_i\rangle_{i\in \omega}$ and a formula $\phi( \bar x; \bar b)$ such that $\phi(\bar a_i; \bar b)$ holds in $M$ if and only if $i$ is odd.
\end{defn}

Let $G$ be a group definable in a dependent theory, let  $H$ be a subgroup of $G$ and let $\G$ be a $|H|^+$-saturated extension of $G$. The following two results summarize what we know about envelopes of $H$. The first was proven by Shelah in \cite{Sh783} and the second by de Aldama in \cite{dAl}.

\begin{fact} If $H$ is abelian, then there exists a definable abelian subgroup of $\G$ which contains $H$.
\end{fact}

\begin{fact} If $H$ is a nilpotent (respectively normal solvable) subgroup of $G$ of class $n$, then there exists a definable nilpotent  (respectively normal solvable) subgroup of $\G$ of class $n$ which contains $H$.
\end{fact}


\bigskip

We now turn to the simple theory context.

\begin{defn}
A theory has the \emph{tree property}  if there exists a formula $\phi(\bar x; \bar y)$, a parameter set $\{\bar a_\mu : \mu \in  \omega^{< \omega}\}$ and $ k \in \omega$ such that
\begin{itemize}
\item  $\{ \phi (\bar x ; \bar a_{\mu^\smallfrown i} : i < \omega \}$ is $k$-inconsistent for any $\mu \in   \omega^{< \omega}$;
\item $\{ \phi (\bar x ;\bar a_{s \upharpoonright n} : s \in \omega^\omega, n \in \omega\}$ is consistent.
\end{itemize}

A theory is called \emph{simple} if it does not have the tree property.
\end{defn}

As the following remark shows, it is impossible to get envelopes in the same way one could achieve them in the stable and dependent case, and one must allow for some ``finite noise''.

\begin{remark}\label{extraspecial}
Let $T$ be the theory of an infinite vector space
over a finite field together with a skew symmetric bilinear form. Then $T$ is simple, and in any model of $T$
one can define an infinite ``extraspecial $p$-group'' $G$, i.\ e.\ every element of $G$ has order $p$, the center of $G$ is cyclic of order $p$ and is equal to the derived group of $G$. This group has  SU-rank 1. It has infinite abelian subgroups but no abelian subgroup of finite index, as the center is finite and any centralizer has finite index in $G$. However, if G had an infinite definable abelian subgroup, that abelian group would have SU-rank 1, hence would be of finite index in G, a contradiction.

A model theoretic study of extra special $p$-groups can be found in \cite{Fe}.
\end{remark}

So one has to find a version of the theorem which is adapted to the new context. For this we will need the following definitions:

\begin{defn}
A group $G$ is called \emph{finite-by-abelian} if there exists a finite normal subgroup $F$ of $G$ such that $G/F$ is abelian.
\end{defn}

\begin{defn}
A subgroup $H$ of a group $G$ is an \emph{almost abelian group} if the centralizer of any of its elements has finite index in $H$. If the index of these elements can be bounded by some natural number we call it an \emph{bounded almost abelian group}.
\end{defn}

Almost abelian groups are also known as FC-groups, where FC-group stands for ``finite conjugation''-group.

The following classical group theoretical result, which is a theorem of Neumann, will provide a link between the two notions.

\begin{fact}\label{fac_neum}  \cite[Theorem 3.1]{bhn}.
Let $G$ be a bounded almost abelian group. Then its derived group is finite. In particular, $G$ is finite-by-abelian.
\end{fact}

Now we are ready to announce the abelian version for simple theories proven by Milliet as \cite[Proposition 5.6.]{Mi}.

\begin{fact}
Let $G$ be a group definable in a simple theory and let  $H$ be an abelian subgroup of $G$. Then there exists a definable finite-by-abelian subgroup of $G$ which contains $H$.
\end{fact}

In the nilpotent and solvable case one must additionally include other definitions to account for the ``by finite'' phenomenon.

\begin{defn}
Let $G$ be a group and $H$ and $K$ be two subgroups of $G$.
%
We say that $H$ is \emph{almost contained} in $K$, denoted by  $H\ls K$,   if $[H : H \cap K]$ is finite.
%
\end{defn}

The following was proved by Milliet  in \cite{Mi2}:
\begin{fact} Let $G$ be a group definable in a simple theory and let $H$ be a nilpotent  (respectively solvable)  subgroup of $G$ of class $n$. Then one can find a definable nilpotent  (respectively solvable)  subgroup of class at most $2n$ which almost contains $H$.
\end{fact}

If we additionally assume that the nilpotent subgroup $H$ is normal in $G$, one can ask for the definable subgroup which almost contains $H$ to be normal in $G$ as well. Hence the product of these two groups is a definable normal nilpotent subgroup of $G$ of class at most $3n$ which \textbf{contains} $H$.


\section{Main result}

The purpose of this paper is to extend the above results to the context of $NTP_2$ theories, which expand both simple and dependent theories.

\begin{defn}\label{NTP2}
A theory has the \textbf{tree property of the second kind} (refered to as TP$_2$) if there exists a formula $\psi(\bar x;\bar y)$, an array of parameters $(\bar a_{i,j} : i, j \in \omega)$, and $k \in \omega$ such that:
\begin{itemize}
\item $\{ \psi(\bar x;\bar a_{i, j}): j\in \omega\}$ is k-inconsistent for every $i \in \omega$;
\item $\{ \psi(\bar x;\bar a_{i, f(i)}): i \in \omega\}$ is consistent for every $f: \omega \rightarrow  \omega$.
\end{itemize}
A theory is called \textbf{NTP$_2$} if it does not have the TP$_2$.
\end{defn}

\begin{observation}
By compactness, having the tree property of the second kind is equivalent to the following finitary version:

A theory has  TP$_2$ if there exists a formula $\psi(\bar x;\bar y)$ and a natural number $k$ such that for any
natural numbers $n$ and $m$ we can find an array of parameters $(\bar a_{i,j} : 1\leq i\leq n ,\  1\leq j\leq m)$ satisfying the following properties:
\begin{itemize}
\item $\{ \psi(\bar x;\bar a_{i, j}): j  \leq m\}$ is k-inconsistent for every $i $;
\item $\{ \psi(\bar x;\bar a_{i, f(i)}): i \leq n\}$ is consistent for every $f: \{1, \dots, n\} \rightarrow  \{1, \dots, m\}$.
\end{itemize}

\end{observation}

In this paper we will prove the following.

\begin{theorem} \label{thm_main}
Let $G$ be a group definable in an NTP$_2$ theory, $H$ be a subgroup of $G$ and $\G$ be an $|H|^+$-saturated extension of $G$. Then the following holds:

\begin{enumerate}
\item If $H$ is abelian, then there exists a definable almost abelian (thus finite-by-abelian) subgroup of $\G$ which contains $H$. Furthermore, if $H$ was normal in $G$, the definable finite-by-abelian subgroup can be chosen to be normal in $G$ as well.\label{thm main 1}

\item If $H$ is a normal solvable subgroup of class $n$, then there exists a definable normal solvable subgroup $S$ of $ \G$ of class at most $2n$ which almost contains $H$. 

\item If $H$ is a nilpotent subgroup of class $n$, then there exists a definable nilpotent subgroup $N$ of $\G$ of class at most $2n$ which almost contains $H$. Moreover, if $H$ is normal in $G$, the group $N$ can be chosen to be normal in $G$ as well.
\end{enumerate}
\end{theorem}

In the abelian and solvable case we follow some of the ideas already present in the proof of de Aldama. Similar to his proof and unlike the proof of Milliet in simple theories, we do not rely on a chain condition for uniformly definable subgroups, but we look to prove the result directly from the non existence of the array described in Definition \ref{NTP2}. In the nilpotent case, we use additionally some properties of the almost centralizer (see Definition \ref{def_AlmCen}) needed to prove the same result in groups which satisfy the chain condition on centralizers up to finite index presented in \cite{Na}.

The following is the key lemma for the abelian case and it is used as well in the nilpotent case.

\begin{lemma}\label{Lem_abl}
Let $G$ be a group with an NTP$_2$ theory and let $H$ be a subgroup of $G$. Fix $\G$ an $|H|^+$-saturated extension of $G$ and let $\phi (x,y)$ be the formula $x \in C_\G(y)$. Consider the following partial types:
$$\pi_{Z(H)}(x) =\{ \phi(x,g): Z(H) \leq \phi(x,g),\ g \in \G\}$$
$$\pi_H(x) =\{ \phi(x,g): H \leq \phi(x,g),\ g \in \G\}.$$
Then there exists a natural numbers $n$ such that
$$ \pi_{Z(H)}(x_0) \cup \dots \cup \pi_{Z(H)}(x_n) \cup \pi_H(y) \vdash \bigvee_{i \neq j} x_i^{-1} x_j \in C_{\G} (y).$$
\end{lemma}

\proof
Suppose that the lemma is false. Then for arbitrary large $n \in \mathbb N$ one can find a sequence of elements $(a_{l,0}, \dots , a_{l,n-1}, b_l)_{l < \omega}$ such that
$$ (\bar a_l , b_l) \models  \pi_{Z(H)}(x_0) \cup \dots \cup \pi_{Z(H)}(x_{n-1}) \cup \pi_H(y) \upharpoonright H \cup \{  \bar a_k , b_k: k < l\}$$
and for all $0\leq i < j < n$ we have that $a_{l,i}^{-1} a_{l,j} \not\in C_{\G} (b_l)$. We show that:
\begin{enumerate}
\item For all $i < n$ and all natural numbers $k$ different than $l$, we have that $a_{l,i} \in C_\G(b_k)$;
\item For all $i, j < n$ and all $k < l< \omega$ we have that $a_{l,i} \in C_\G(b_k^{a_{k,j}})$.
\end{enumerate}

To do so, we let $k<l < \omega$ and $i,j < n$ be arbitrary  and we prove that $a_{l,i} \in C_\G(b_k)$ as well as $a_{k,i} \in C_\G(b_l)$ and  $a_{l,i} \in C_\G(b_k^{a_{k,j}})$.

Let $z$ be an element of $Z(H)$. Hence $H$ is a subgroup of $C_{\G} (z)$ and whence  $\phi(x,z) \in \pi_{H}(x)\upharpoonright H$. As $b_k$ satisfies this partial type, we obtain that
$$Z(H) \leq C_\G(b_k).$$
So $\phi(x,b_k)$ belongs to $\pi_{Z(H)}(x) \upharpoonright \{b_k\}$. Since the element $a_{l,i}$ satisfies $\pi(x)_{Z(H)} \upharpoonright H \cup \{b_k\}$, we get that $a_{l,i}$ belongs to $ C_\G(b_k)$.

On the other hand, if we take $a \in H$ we have that $Z(H)$ is a a subgroup of $C_{\G} (a)$ and thus  $\phi(x,a) \in \pi_{Z(H)}(x)\upharpoonright H$. As $a_{k,i}$ satisfy this partial type, we obtain that
$$H \leq C_\G(a_{k,i}).$$
So $\phi(x,a_{k,i}) \in \pi_{H}(x) \upharpoonright \{a_{k,i}\}$. As the element $b_l$ satisfies this partial type $\pi(x)_H \upharpoonright H \cup \{a_{k,i}\}$, we get that  the element $a_{k,i}$ belongs to $C_\G(b_l)$ which together with the previous paragraph yields (1).

As seen before, we have that $Z(H) \leq C_\G(b_k)$ and $H \leq C_\G(a_{k,i})$. This yields that  $Z(H) \leq C_\G(b_k^{a_{k,j}})$. Hence  $\phi(x,b_k^{a_{k,j}})$ belongs to $\pi_{Z(H)}(x) \upharpoonright \{b_k, a_{k,j}\}$. Since the element $a_{l,i}$ satisfies $\pi(x)_{Z(H)} \upharpoonright H \cup \{b_k, a_{k,j}\}$, we obtain that $a_{l,i}$ belongs  to $ C_\G(b_k^{a_{k,j}})$ which yields (2).

Let $\psi (x, y, z)$ be the formula that defines the coset of $y\cdot C_\G(z)$. We claim that the following holds:
\begin{itemize}
\item $\{ \psi(x, a_{l, i} ,b_l\} : i < n\}$ is $2$-inconsistent for any $l\in \omega$;
\item $\{ \psi(x, a_{l, f(l)} ,b_l\} : l \in \omega\}$ is consistent for any function $f: \omega \rightarrow n+1$.
\end{itemize}
The first family is $2$-inconsistent as every formula defines a different coset of $C_\G(b_i)$ in $\G$. For the second we have to show that for all natural numbers $m$ and all tuples $(i_0, \dots, i_m) \in n^m$ the intersection $$ a_{0,{i_0}} C_\G(b_0) \cap \dots \cap a_{m,{i_m}} C_\G(b_m)$$ is nonempty. Using (1) and (2) and multiplying by $ a_{0, i_0} \cdot \dots \cdot a_{m, i_m}$ on the right, it is equivalent to $  C_\G(b_0^{a_{0,{i_0}}}) \cap \dots \cap C_\G(b_m^{a_{m,{i_m}}})$ being nonempty which is trivial true.

Compactness yields a contradiction to the fact that the group $G$ has an NTP$_2$ theory and we obtain the result.
\qed

\subsection{Abelian subgroups}


\proof[Proof of Theorem \ref{thm_main}(1)]
Since $H$ is abelian, it is equal to its center. So by Lemma \ref{Lem_abl} and compactness one can find a finite conjunction $\bigwedge_{i} \phi(x,g_i)$ with $\phi (x,y)$ being the formula $x \in C_\G(y)$ and $g_i$ in some saturated extension of $G$, such that
$$\bigwedge_{i} \phi(x_0,g_i) \wedge \dots \wedge \bigwedge_{i} \phi(x_n,g_i) \wedge \bigwedge_{i} \phi(y,g_i) \vdash x_i^{-1} x_j \in C_{\G} (y). \ \ \ (\ast)$$
Furthermore, all $h$ in $H$ satisfies $\bigwedge_{i} \phi(x,g_i)$.
Hence the subgroup $ \bigcap_i C_G(g_i)$ contains $H$ and by ($\ast$), it is a bounded almost abelian group. Thus its commutator subgroup is finite by Fact \ref{fac_neum}, which yields Theorem \ref{thm_main}$(1)$. Moreover, if $H$ is normal in $G$, the group $ \bigcap_i (C_G(g_i)^G)$ is a definable normal subgroup of $G$ which still contains $H$ and which is as well finite-by-abelian, which completes the proof.
\qed

\subsection{Solvable subgroups}

To prove the solvable case of Theorem \ref{thm_main} we need the following.
\begin{defn}
A group $G$ is \emph{almost solvable} if there exists a normal \emph{almost series} of finite length, i.\,e.\ a finite sequence of normal subgroups
$$\{1\} = G_0 \trianglelefteq G_1 \trianglelefteq \dots\trianglelefteq G_n = G$$
of $G$ such that $G_{i+1} /G_{i}$ is an almost abelian group for all $i \in n$. The least such natural number $n \in \omega$ is called the \emph{almost solvable class} of $G$.
\end{defn}

\begin{defn}
Let $G$ be a group and $S$ be a definable almost solvable subgroup of class $n$. We say that $S$ \emph{admits a definable almost series} if there exists a family of definable normal subgroups $\{S_i: i \leq n\}$ of $S$ such that $S_0$ is the trivial group, $S_n$ is equal to $S$ and $S_{i+1} / S_i$ is almost abelian.
\end{defn}

The proof of Corollary 4.10 in \cite{Mi} provides the following fact. It can also been found as Lemma 3.22 in \cite{Na}.

\begin{fact}\label{lem_FCsol2}
Let $G$ be a definable almost solvable subgroup of class $n$ which admits a definable almost series. Then $G$ has a definable subgroup of finite index which is solvable of class at most $2n$. 
\end{fact}

So we only need to concentrate on building a definable almost series.

\begin{prop}\label{Prop_AlmSol}
Let $G$ be a group definable in an NTP$_2$ theory and $H$ be a normal solvable subgroup of $G$ of class $n$. Then there exists a definable normal almost solvable subgroup $S$ of class $n$ containing $H$. Additionally, $S$ admits a definable almost series such that all of its members are normal in $G$.
\end{prop}

\proof
We prove  this by induction on the derived length $n$ of $H$. If $n$ is equal to $1$ this is a consequence of the abelian case, Theorem \ref{thm_main}$(1)$. So let $n>1$, and consider the abelian subgroup $H^{(n-1)}$ of $H$. It is a characteristic subgroup of $H$ and hence, as $H$ is normal in $G$, it is normal in $G$ as well. So again by the abelian case, there exists a definable almost abelian normal subgroup $A$ of $G$ which contains $H^{(n-1)}$. Replacing $G$ by $G/A$, we have that the derived length of $HA/A$ is at most $n-1$ and we may apply the induction hypothesis which finishes the proof.
\qed

\begin{proof}[Proof of Theorem \ref{thm_main}(2)]
Applying Proposition \ref{Prop_AlmSol} to $H$ gives us a definable almost solvable group $K$ of class $n$ containing $H$ and which admits a definable almost series. By Fact \ref{lem_FCsol2}, the group $K$ has a definable subgroup $S$ of finite index which is solvable of class at most $2n$. 
\end{proof}

\subsection{Nilpotent subgroups}

The following follows from Lemma \ref{Lem_abl}
\begin{lemma}\label{Cor_nil}
Let $G$ be a group definable in an NTP$_2$ theory, let $H$ be a subgroup of $G$ and suppose that $G$ is $|H|^+$ saturated. Then one can find definable subgroups $A$ and $K$ and a natural number $m$ such that
\begin{itemize}
\item  the cardinality of the conjugacy class $k^A$ for all elements $k$ in $K$ is bounded by $m$.
\item $A$ is almost abelian and contains $Z(H)$;
\item $K$ contains $H$ and $A$.
\end{itemize}
If $H$ is additionally normal in $G$, one can choose $A$ and $K$ to be normal in $G$ as well.
\end{lemma}

\proof
By Lemma \ref{Lem_abl}  we can find $\phi_{Z(H)}$ and $\phi_H$ which are conjunctions of formulas from $\pi_{Z(H)}(x)$ and $\pi_H(x)$ (defined as in Lemma \ref{Lem_abl}) respectively and a natural number $m$ such that:
$$ \phi_{Z(H)}(x_0) \wedge \dots \wedge  \phi_{Z(H)}(x_m) \wedge  \phi_H(y) \vdash \bigvee_{i \neq j} x_i^{-1} x_j \in C_{\G} (y).$$
Note that these formulas define intersections of centralizers and are therefore subgroups of $G$. Letting $A$ be equal to  $\phi_{Z(H)}(G) \cap \phi_H (G)$ and $K$ be equal to $\phi_H(G)$ we have the announced properties.

If $H$ is normal in $G$, we have that $Z(H)$ is also normal in $G$ and we can replace $A$ and $K$ by $\cap_{g\in G} A^g$ and $\cap_{g\in G} K^g$ which are normal definable subgroups of $G$ and still satisfy the given properties.
\qed

To prove the existence of ``definable envelopes'' of nilpotent subgroups of a group definable in an NTP$_2$ theory we need to define the almost-centralizer.

\begin{defn}\label{def_AlmCen}
Let $G$ be a group an $H$ be a definable subgroup  of $G$.
 We define the almost-centralizer $\widetilde{C}_G(H)$ of $H$ in $G$  to be
\[\widetilde{C}_G(H):=\{g\in G \mid [H : C_H(g)] < \infty \}.\]
\end{defn}

We will need the following results, which are Corollary 2.11 and Proposition 3.27 in \cite{Na}.

\begin{fact}\label{Fac_sym} \textbf{(Symmetry)}
Let $G$ be a group and let $H$ and $K$ be two definable subgroups of $G$. So
$$ H \ls \t C_G(K) \mbox{ if and only if } K \ls \t C_G(H).$$
\end{fact}

\begin{fact}\label{Fac_[H,K]fin}
Let $G$ be a group and let $H$ and $K$ be two definable subgroups of $G$ such that $H$ is normalized by $K$. Suppose that $H$ is contained in $\t C_G(K)$ and $K$ is contained in $\t C_G(H)$. So $[H,K]$ is finite.
\end{fact}

We will also need a theorem which is the definable version of a result proven by Schlichting in \cite{Sch} and which can be found in \cite{Wag_book} as Theorem 4.2.4.  It deals with families of uniformly commensurable subgroups, a notion we now introduce:

\begin{defn}
A family $\mathcal{F}$ of subgroups is \emph{uniformly commensurable} if there exists a natural number $d$ such that for each pair of groups $H$ and $K$ from $\mathcal{F}$ the index of their intersection is smaller than $d$ in both $H$ and $K$.
\end{defn}

\begin{fact}\label{fac_schl}
Let $G$ be a group and $\mathcal H$ be a family of definable uniformly commensurable subgroups. Then there exists a definable subgroup $N$ of $G$ which is commensurable which all elements of $\mathcal H$ and which is invariant under any automorphisms of $G$ which stabilizes $\mathcal H$ setwise.
\end{fact}

\proof[Proof of Theorem \ref{thm_main}(3)]
Note that if $H$ is finite the result holds trivially. So we may assume that $H$ is infinite and suppose as well that $G$ is already  $|H|^+$-saturated.

We will prove by induction on the nilpotency class $n$ of $H$ that there exists a definable nilpotent subgroup $N$ of $G$ of class at most $2n$ and a sequence of subgroups:
$$\{1\} = N_0 \leq N_1 \leq N_2 \leq \dots \leq N_{2n} = N$$
such that  $H\ls N$ and for all $0\leq i < 2n$, we have that
\begin{itemize}
\item $N_i$ is definable and normal in $N$;
\item $[N_{i+1}, N]\leq N_i$.
\end{itemize}

If $H$ was supposed to be normal in $G$, we asked each $N_i$ to be normal in $G$ as well.

Let $n$ be equal to $1$. Then $H$ is abelian, and by Theorem \ref{thm_main}$(1)$ there exists a definable almost abelian subgroup $A$ of $G$ which contains $H$. As $[A,A]$ is finite by Fact \ref{fac_neum}, letting $N$ be equal to $C_A([A,A])$ and $N_1 = Z([A,A])$ gives the desired groups.

Assume now $n$ is strictly greater than $1$. We first show the following:

\begin{claim}
There is some definable subgroups $A$ and $K$ of $G$ such that:
\begin{itemize}
\item $A$ is a normal subgroup of $K$;

\item $Z(H)\ls A$ and $H \leq K$;

\item $K\leq \t C_G(A)$ and $ A \leq \t C_G(K)$;

\item $[A,K]$ is finite and contained in $\t C_G(K)$.

\end{itemize}

\end{claim}

\begin{proof}
First, by Lemma \ref{Cor_nil}  we can find definable subgroups $A_0$ and $K$ of $G$  such that
\begin{enumerate}
\item   the cardinality of the conjugacy class $k^{A_0}$ for all elements $k$ in $K$ is bounded by  some natural number $m$;
\item $A_0$ is almost abelian and contains $Z(H)$;
\item $K$ contains $H$ and $A_0$.
\end{enumerate}

The next step to prove the claim is to replace $A_0$ by a commensurable subgroup which is additionally normal in $K$.

By (1) we can deduce that $K$ is contained in $\t C_G(A_0)$ and therefore $$\mathcal F =\{ {A_0}^k : k \in K\}$$ is a uniformly definable and uniformly commensurable family of subgroups of $G$. By Schlichting (Fact \ref{fac_schl})  one can find a definable subgroup $A_1$ of $K$ which is commensurable with all groups in $\mathcal F$, in particular with $A_0$, and which is stabilized by all automorphisms which stabilize the family setwise, and thus is normal in $K$.

The third step is to find a subgroup which is commensurable with $A_0$, still normal in $K$ and additionally contained in $\t C_K(K)$.
As $A_1$ is commensurable with $A_0$, we have that $K \leq \t C_K(A_1)$. By symmetry of the almost centralizer, we obtain that $A_1$ is almost contained in $\t C_K(K)$.  Let $A=A_1\cap K$; this is still a normal subgroup of $K$ and has finite index in $A_1$. Now, we only need to prove that it is definable.

Since it has finite index, the definable subgroup $A_1$ is a finite union of distinct cosets of $A$, say $A_1 = \bigcup _{i=1}^k a_i A$ for some $a_i\in A_1$. Furthermore, we have that $A$ is the union of the definable sets
\[
A_d:=\phi_d (x) = \{ x \in A_1 : [ K: C_K (x) ] < d\}.
\]
But then we have that
\[A_1=\bigcup _{i=1}^k \bigcup_{d\in \mathbb N} a_i A_d\]
so by compactness and saturation of $G$ this is equals to a finite subunion. Additionally, as  $\{A_d\}_{d\in \omega}$ was a chain of subsets of $A$ each contained in the next we have that
\[A_1=\bigcup _{i=1}^k  a_i A_d\]
for some fixed $d$. Hence $A$ is equal to $A_d$ and whence it  is a normal definable subgroup of $K$. Moreover, the group $A$ is commensurable with $A_0$, so it almost contains $Z(H)$  and $K$ is still in $\t C_G(A)$. Additionally, $A$ is contained in $\t C_G(K)$ and normal in $K$. Hence, by Fact \ref{Fac_[H,K]fin}, we have that the group $[A, K]$ is finite and as $A$ is normal in $K$, it is contained in $A$ and in $\t C_G(K)$. \end{proof}

%
%
%
%

\bigskip
Let $A$ and $K$ be as in the claim, so the index $[Z(H): A \cap Z(H)]$ is finite. Take a set of representative of coclasses $\{h_0, \dots, h_n\}$ of $A \cap Z(H)$ in $Z(H)$ and set $K'$ to be $C_K(h_0, \dots, h_n)$ and $A'$ to be $A \cap K'$. Note that $Z(H) \cap A$ equals $Z(H) \cap A'$. So, we still have the simililar properties for $A'$ and $K'$ as for $A$ and $K$, namely:
\begin{itemize}

\item $ H \leq K'$ and $Z (H) \ls A'$, in particular  $Z(H) \cap A \leq A'$;
\item $[A', K']$ is finite and contained in $\t C_G(K')$.

\end{itemize}

Let $X$ be equal to the finite subgroup $[A',K']$ of $K'$. Define $N^\ast$ to be $ C_{K'}(X)$, which is,  as $X$ is contained in $\t C_G(K')$, a definable subgroup of $K'$ of finite index. Let $N_1$ be equal to the definable normal subgroup $Z(N^\ast)$ which contains $X$ and $\{h_0, \dots, h_n\} \cap N^\ast$, and let $N_2$ be equal to the definable normal subgroup $Z_2(N^\ast)$  of $N^\ast$ which contains $A' \cap N^\ast $, still contains $\{h_0, \dots, h_n\}\cap N^\ast$ and therefore contains $Z(H) \cap N^\ast$. The groups $N_1$ and $N_2$ are the first two members of the nilpotent sequence we need to define, and they allow us to use the induction hypothesis.

\bigskip

Since $N_2$ contains $Z(H) \cap N^\ast$, we have that the group $(H\cap N^\ast) N_2/N_2$ has nilpotency class strictly smaller than $H$ and is contained in the group $N^\ast /N_2$ which is a interpretable group in an NTP$_2$ theory and thus NTP$_2$. We may apply the induction hypothesis to this group and find a definable sequence of subgroups of $N^\ast/N_2$
$$N_2/N_2 \leq N_3/N_2 \leq \dots \leq N_{2n}/N_2$$
such that
\begin{itemize}
\item each element of the sequence is normal in $N_{2n}/N_2 $,
\item  $H\cap N^\ast/N_2$ is almost contained in $N_{2n}/N_2$, and
\item  for all $i>1$ $N_{i+1}/N_i$ is in the center of $N_{2n}/N_i$.
\end{itemize}
As $N_{2n}$ is a subgroup of  $N^\ast$ we have that  $N_1 \leq Z(N_{2n})$ and $[N_2, N_{2n}] \leq N_1$. Moreover, $H \cap N^\ast$ has finite index in $H$ and hence $H$ is as well almost contained in $N_{2n}$ almost contains $H$. Whence,
$$\{1\} = N_0 \leq N_1 \leq \dots \leq N_{2n}$$
is an ascending central series of $N_{2n}$ with the desired properties.

\bigskip

Now, suppose that $H$ is normal in $G$. We consider the groups:
$$L := K^G \cap C_G( h_0^G, \dots , h_n^G) \ \ \ \mbox{ and } \ \ \ B = A^G\cap L$$
These are now normal subgroups of $G$ and we have as well that:
\begin{itemize}
\item $ H \leq L$ and $Z (H) \ls B$, in particular  $Z(H) \cap A \leq B$;
\item $[B, L]$ is finite and contained in $\t C_G(L)$.
\end{itemize}
Doing the same construction to find $N_1$ and $N_2$ replacing $K'$ and $A'$ by $L$ and $B$, we have additionally that $N_1$ and $N_2$ are normal in $G$. To find the rest of the sequence, we can finish as above using the fact that by induction hypothesis all groups will be normal in $G$.
\qed

\begin{cor}
Let $G$ be a group with an NTP$_2$ theory and let $H$ be a nilpotent normal subgroup of $G$ of class $n$. Then there exists a definable nilpotent normal subgroup $N$ of $G$ of class at most $3n$ which contains $H$.
\end{cor}
\begin{proof}
By Theorem \ref{thm_main}(3), there is a definable nilpotent subgroup $N_0$ of class at most $2n$ which almost contains $H$ and which is normal in $G$. Thus, the group $NH$ is a definable normal nilpotent subgroup of $G$ of nilpotency class at most $3n$.
\end{proof}

\bibliographystyle{abbrv}

\bibliography{groups}

\end{document}